\documentclass[a4paper, 11pt]{amsart}

\usepackage[T1]{fontenc}
\usepackage[english]{babel}
\usepackage[utf8x]{inputenc}
\usepackage{hyperref}
\usepackage{setspace}
\usepackage{bbm}
\usepackage[dvips]{graphicx}
\usepackage{lmodern}
\usepackage{paralist}
\usepackage{wrapfig}
\usepackage{float}
\usepackage[usenames,dvipsnames]{pstricks}
\usepackage{epsfig}
\usepackage{pst-grad} % For gradients
\usepackage{pst-plot} % For axes
\usepackage{amsmath}
\usepackage{amssymb}
\usepackage{wasysym}
\newcommand{\INT}[2]{\int\limits_{#1}^{#2} } %Schöne Integrale
\renewcommand{\epsilon}{\varepsilon}
\renewcommand{\theta}{\vartheta}
\renewcommand{\phi}{\varphi}
\renewcommand{\leq}{\leqslant}
\renewcommand{\geq}{\geqslant}
\newcommand{\D}{\ \mathrm{d}}
\newcommand{\R}{\mathbb{R}}

\newcommand{\I}{\mathrm{i}}
\newcommand{\1}{\mathbbm{1}}
\renewcommand{\d}{\mathrm{d}}

\newcommand{\e}{\mathrm{e}}

\newcommand{\EXP}[1]{\exp\left(#1\right)}

\newcommand{\VAR}{\mathbb{V}}
\newcommand{\beweisende}{\hfill $\square$}
\DeclareMathOperator{\RE}{Re}

\newcommand{\Cov}{\mathop{\mathrm{Cov}}\nolimits}
\newtheorem{theorem}{Theorem}
\newtheorem{lemma}{Lemma}

\usepackage{amsthm}

\renewcommand{\P}[1]{\mathbb{P}\left(#1\right)} % Wahrscheinlichkeitsmaß

\newcommand{\N}{\mathbb{N}}
\newcommand{\E}{\mathbb{E}}

\newcommand{\IND}[1]{\1_{#1}}
\newcommand{\EW}[1]{\E\left[#1\right]}

\DeclareMathOperator{\sgn}{sgn}

\newcommand{\Z}{\mathbb{Z}}

\author{Hendrik Flasche}
\address{Hendrik Flasche, Institut f\"ur Mathematische Statistik,
Universit\"at M\"unster,
Orl\'eans--Ring 10,
48149 M\"unster, Germany}

\email{{h\_flas01@uni-muenster.de}}

\title[Real roots of random trigonometric polynomials]{Expected number of real roots of random trigonometric polynomials}

\begin{document}

\begin{abstract}
We investigate the asymptotics of the expected number of real roots of random trigonometric polynomials
$$
X_n(t)=u+\frac{1}{\sqrt{n}}\sum_{k=1}^n (A_k\cos(kt)+B_k\sin(kt)), \quad t\in [0,2\pi],\quad u\in\R
$$
whose coefficients $A_k, B_k$, $k\in\N$, are independent identically distributed random variables with zero mean and unit variance. If $N_n[a, b]$ denotes the number of real roots of $X_n$ in an interval $[a,b]\subseteq [0,2\pi]$, we prove that
$$
\lim_{n\rightarrow\infty} \frac{\E N_n[a,b]}{n}=\frac{b-a}{\pi\sqrt{3}}
\EXP{-\frac{u^2}{2}}.
$$
\end{abstract}
\maketitle

%-----------------------------------------------------------------------------------------------------------------------
\section{Introduction}
\subsection{Main result}
In this paper we are interested in the number of real roots of a random trigonometric polynomial $X_n:[0,2\pi]\rightarrow \R$ defined as
\begin{equation}\label{eq:trigpolynomial}
X_n(t):=u+\frac{1}{\sqrt{n}}\sum_{k=1}^n (A_k\cos(kt)+B_k\sin(kt)),
\end{equation}
where $n\in \N$, $u\in\R$, and the coefficients $(A_k)_{k\in\N}$ and $(B_k)_{k\in \N}$  are independent identically distributed random variables with
\begin{equation}\label{eq:exp_var}
\E A_k = \E B_k = 0, \quad \E [A_k^2] = \E [B_k^2] =1.
\end{equation}
The random variable which counts the number of real roots of $X_n$ in an interval $[a,b]\subseteq [0,2\pi]$ is denoted by $N_n[a,b]$. By convention, the roots are counted with multiplicities and a root at $a$ or $b$ is counted with weight  $1/2$.
The main result of this paper is as follows.
%%%%%%%%%%%%%%%%%%%%%%%%%%%%%%%%%%%%%%%%%%%%%%%%%%%%%%%%%%%%%%%%%%%%%%%%%%%%%%%%%%%%%%%%%%%%%%%%%%%
\begin{theorem}\label{satz:hauptaussage}
Under assumption~\eqref{eq:exp_var} and for arbitrary $0\leq a<b\leq 2\pi$, the expected number of real roots of $X_n$ satisfies
\begin{equation}\label{eq:hauptaussage}
\lim_{n\rightarrow\infty} \frac{\E N_n[a,b]}{n}=\frac{b-a}{\pi\sqrt{3}}
\EXP{-\frac{u^2}{2}}.
\end{equation}
\end{theorem}
The number of real roots of random trigonometric polynomials has been much studied in the case when the coefficients $A_k, B_k$ are Gaussian; see~\cite{dunnage}, \cite{das},  \cite{qualls}, \cite{wilkins}, \cite{farahmand1}, \cite{sambandham1}, to mention only few references, and the books~\cite{farahmand_book}, \cite{bharucha_reid_book}, where further references can be found. In particular, a proof of~\eqref{eq:hauptaussage} in the Gaussian case can be found in~\cite{dunnage}. Recently, a central limit theorem for the number of real roots was obtained in~\cite{granville_wigman} and then, by a different method employing Wiener chaos expansions, in~\cite{azais_leon}.  For random trigonometric polynomials involving only cosines, the asymptotics for the variance (again, only in the Gaussian case) was obtained in~\cite{su_shao}.

All references mentioned above rely heavily on the Gaussian assumption which allows for explicit computations. Much less is known when the coefficients are non-Gaussian. In the case when the coefficients are uniform on $[-1,1]$ and there are no terms involving the sine, an analogue of~\eqref{eq:hauptaussage} was obtained in~\cite{sambandham}. The case when the third moment of the coefficients is finite, has been studied in~\cite{sambandham_thangaraj}.  After the main part of this work was completed, we became aware of the work of Jamrom~\cite{jamrom} and a recent paper by Angst and Poly~\cite{angst_poly}.  Angst and Poly~\cite{angst_poly}  proved~\eqref{eq:hauptaussage} (with $u=0$) assuming that the coefficients $A_k$ and $B_k$ have finite $5$-th moment and satisfy certain Cram\'er-type condition. Although this condition is satisfied by some discrete probability distributions, it excludes the very natural case of $\pm1$-valued Bernoulli random variables. Another recent work by Aza\"is et.\ al.~\cite{azais_etal} studies the local distribution of zeros of random trigonometric polynomials and also involves conditions stronger than just the existence of the variance.  In the paper of Jamrom~\cite{jamrom}, Theorem~\ref{satz:hauptaussage} (and even its generalization to coefficients from an $\alpha$-stable domain of attraction) is stated without proof. Since full details of Jamrom's proof do not seem to be available and since there were at least three works following~\cite{jamrom} in which the result was established under more restrictive conditions (namely, \cite{sambandham}, \cite{sambandham_thangaraj}, \cite{angst_poly}), it seems of interest to provide a full proof of Theorem~\ref{satz:hauptaussage}.

\subsection{Method of proof}
The proof uses ideas introduced by Ibragimov and Maslova~\cite{ibragimov_maslova1} (see also the paper by Erd\"os and Offord~\cite{erdoes_offord}) who studied  the expected number of real zeros of a random algebraic polynomial of the form
\begin{equation*}
Q_n(t):=\sum_{k=1}^n A_k t^k.
\end{equation*}

For an interval $[a,b]\subset[0,2\pi]$ and $n\in\N$ we introduce the random variable $N_n^*[a,b]$ which is the indicator of a \emph{sign change} of $X_n$ on the endpoints of $[a,b]$ and is more precisely defined as follows:
\begin{equation}\label{eq:definitionnstern}
N_n^*[a,b]:=\frac{1}{2}-\frac{1}{2}\sgn(X_n(a)X_n(b))=\begin{cases}
0 &\textrm{if } X_n(a)X_n(b)>0, \\
1/2 &\textrm{if } X_n(a)X_n(b)=0, \\
1 &\textrm{if } X_n(a)X_n(b)<0.
\end{cases}
\end{equation}
The proof of Theorem \ref{satz:hauptaussage} consists of two main steps.

\vspace*{2mm}
\noindent
\textit{Step 1: Reduce the study of roots to the study of sign changes.}
Intuition tells us that $N_n[\alpha,\beta]$ and $N_n^*[\alpha,\beta]$ should not differ much if the interval $[\alpha,\beta]$ becomes small.  More concretely, one expects that the number of real zeros of $X_n$ on $[0,2\pi]$ should be of order $n$, hence the distance between consecutive roots should be of order $1/n$. This suggests that on an interval $[\alpha,\beta]$ of length $\delta n^{-1}$ (with small $\delta>0$) the event of having at least two roots  (or a root with multiplicity at least $2$) should be very unprobable. The corresponding estimate will be given in Lemma~\ref{lemma:abschaetzungdmj}. For this reason, it seems plausible that on intervals of length $\delta n^{-1}$ the events ``there is at least one root'', ``there is exactly one root'' and ``there is a sign change'' should almost coincide. A precise statement will be given in Lemma~\ref{lemma:ewuntnsternn}. This part of the proof relies heavily on the techniques introduced by Ibragimov and Maslova~\cite{ibragimov_maslova1} in the case of algebraic polynomials.

\vspace*{2mm}
\noindent
\textit{Step 2: Count sign changes.} %First we give an estimate for $\E N_n[\alpha,\beta]-\E N^*_n[\alpha,\beta]$ for large $n\in \N$ on small intervals $[\alpha,\beta]\subset [0,2\pi]$ of length $\beta-\alpha=n^{-1}\delta$ for a small $\delta>0$  (section \ref{abschnitt:unterschied}).
We compute the limit of $\E N_n^*[\alpha_n,\beta_n]$ on an interval $[\alpha_n,\beta_n]$ of length $\delta n^{-1}$. This is done by establishing a bivariate central limit theorem stating that as $n\to\infty$ the random vector $(X_n(\alpha_n), X_n(\beta_n))$ converges in distribution to a Gaussian random vector with mean $(u,u)$, unit variance, and covariance $\delta^{-1}\sin\delta$. From this we conclude that $\E N_n^*[\alpha_n,\beta_n]$ converges to the probability of a sign change of this Gaussian vector. Approximating the interval $[a,b]$ by a lattice with mesh size $\delta n^{-1}$ and passing to the limits $n\to\infty$ and then $\delta\downarrow 0$ completes the proof. This part of the proof is much simpler than the corresponding argument of Ibragimov and Maslova~\cite{ibragimov_maslova1}.

\vspace*{2mm}
\noindent
\textit{Notation.}
The common characteristic function of the random variables $(A_k)_{k\in \N}$ and $(B_k)_{k\in\N}$ is denoted by
$$
\phi(t):=\E \EXP{\I t A_1}, \quad t\in\R.
$$
Due to the assumptions on the coefficients in~\eqref{eq:trigpolynomial},  we can write
\begin{equation}\label{eq:charfuncchar}
\phi(t)=\EXP{-\frac{t^2}{2}H(t)}
\end{equation}
for sufficiently small $|t|$, where $H$ is a continuous function with $H(0)=1$.

In what follows, $C$ denotes a generic positive constant which may change from line to line.

%-----------------------------------------------------------------------------------------------------------------------
\section{Estimate for $\E N_n[a,b]-\E N_n^*[a,b]$ on small intervals}\label{abschnitt:unterschied}

In this section we investigate the expected difference between $N_n[\alpha,\beta]$ and $N_n^*[\alpha,\beta]$ on small intervals $[\alpha,\beta]$ of length $n^{-1} \delta $, where $\delta>0$ is fixed.

\subsection{Expectation and variance}
The following lemma will be frequently needed.
%%%%%%%%%%%%%%%%%%%%%%%%%%%%%%%%%%%%%%%%%%%%%%%%%%%%%%%%%%%%%%%%%%%%%%%%%%%%%%%%%%%%%%%%%%%%%%%%%%%
\begin{lemma}\label{lemma:varianzderableitung}
For $j\in\N_0$ let $X_n^{(j)}(t)$ denote the $j$th derivative of $X_n(t)$.
The expectation and the variance of $X_n^{(j)}$ are given by
\begin{equation*}
\E X_n^{(j)}(t)=
\begin{cases}
u ,&j=0,\\
0 ,&j\in\N,
\end{cases}
\quad \quad
\VAR{X^{(j)}_n(t)}
=
\frac{1}{n}\sum_{k=1}^n k^{2j}.
\end{equation*}
\end{lemma}
%--------------------------------------------------------------------------------------------------
\begin{proof}
The $j$th derivative of $X_n$ reads as follows:
\begin{align*}
&\quad X_n^{(j)}(t)-u\IND{j=0} \\
&=\frac{1}{\sqrt{n}}\sum_{k=1}^n \left(A_k\frac{\d^j}{\d t^j}\cos(kt)
+B_k\frac{\d^j}{\d t^j}\sin(kt)\right) \\
&=
\frac{1}{\sqrt{n}}\sum_{k=1}^n k^j
\begin{cases}
(-1)^{j/2}A_k\cos(kt)+(-1)^{j/2}B_k\sin(kt), &\textrm{if }j \textrm{ is even,} \\
(-1)^{\frac{j-1}{2}}A_k\sin(kt)+(-1)^{\frac{j-1}{2}}B_k\cos(kt), &\textrm{if }j \textrm{ is odd.}
\end{cases}
\end{align*}
Recalling that $(A_k)_{k\in \N}$ and $(B_k)_{k\in\N}$ have zero mean and unit variance we immediately obtain the required formula.
\end{proof}
%%%%%%%%%%%%%%%%%%%%%%%%%%%%%%%%%%%%%%%%%%%%%%%%%%%%%%%%%%%%%%%%%%%%%%%%%%%%%%%%%%%%%%%%%%%%%%%%%%%

\subsection{Estimate for the probability that $X_n^{(j)}$ has many roots}\label{subsec:D_m_j}
Given any interval $[\alpha,\beta]\subset [0,2\pi]$, denote by $D^{(j)}_m=D^{(j)}_m(n; \alpha, \beta)$  the event that the $j$th derivative of $X_n(t)$ has at least $m$ roots in $[\alpha,\beta]$ (the roots are counted with their multiplicities and the roots on the boundary are counted without the weight $1/2$). Here, $j\in \N_0$ and $m\in\N$. A key element in our proofs is an estimate for the probability of this event presented in the next lemma.

%%%%%%%%%%%%%%%%%%%%%%%%%%%%%%%%%%%%%%%%%%%%%%%%%%%%%%%%%%%%%%%%%%%%%%%%%%%%%%%%%%%%%%%%%%%%%%%%%%%
\begin{lemma}\label{lemma:abschaetzungdmj}
Fix $j\in \N_0$ and $m\in\N$. For $\delta>0$ and $n\in\N$ let $[\alpha,\beta]\subset [0,2\pi]$ be any interval of length $\beta-\alpha=n^{-1}\delta$.
Then,
\begin{equation*}
\P{D_m^{(j)}} \leq  C(\delta^{(2/3)m}
+\delta^{-(1/3)m}n^{-(2j+1)/4}),
\end{equation*}
where  $C=C(j,m)>0$ is a constant independent of $n$, $\delta$, $\alpha$, $\beta$.
\end{lemma}
%%%%%%%%%%%%%%%%%%%%%%%%%%%%%%%%%%%%%%%%%%%%%%%%%%%%%%%%%%%%%%%%%%%%%%%%%%%%%%%%%%%%%%%%%%%%%%%%%%%
\begin{proof}
For arbitrary $T>0$ we may write
\begin{align*}
\P{D^{(j)}_m}\leq
\P{D^{(j)} _m\cap\left\{\left|\frac{X^{(j)}_n(\beta)}{n^j}\right|\geq T\right\}}
+
\P{\left|\frac{X^{(j)}_n(\beta)}{n^j}\right|< T}.
\end{align*}
The terms on the right-hand side will be estimated in Lemmas \ref{lemma:abschaetzungzwei} and \ref{lemma:abschaetzungdrei} below.
Using these lemmas, we obtain
\begin{align*}
\P{D^{(j)}_m}
\leq
C\left[\frac{n^m}{T}\frac{(\beta-\alpha)^m}{m!}\right]^2 +
C\left(T+T^{-1/2}n^{-(2j+1)/4}\right).
\end{align*}
Setting $T=\delta^{(2/3)m}$ yields the statement.
\end{proof}
%%%%%%%%%%%%%%%%%%%%%%%%%%%%%%%%%%%%%%%%%%%%%%%%%%%%%%%%%%%%%%%%%%%%%%%%%%%%%%%%%%%%%%%%%%%%%%%%%%%
\begin{lemma}\label{lemma:abschaetzungzwei}
%Adapt notation of Lemma \ref{lemma:abschaetzungdmj}.
For all $j\in\N_0$, $m\in\N$ there exists a constant $C=C(j,m)>0$ such that the estimate
\begin{equation*}
\P{D^{(j)} _m\cap\left\{\left|\frac{X^{(j)}_n(\beta)}{n^j}\right|\geq T\right\}}\leq C\left[\frac{n^m}{T}\frac{(\beta-\alpha)^m}{m!}\right]^2
\end{equation*}
holds for all $T>0$, $n\in \N$ and all intervals $[\alpha,\beta]\subseteq [0,2\pi]$.
\end{lemma}
% % % % % % % % % % % % % % % % % % % % % % % % % % % % % % % % % % % % % % % %
\begin{proof}
By Rolle's theorem, on the event $D^{(j)}_m$  we can find (random) $t_0\geq \ldots \geq t_{m-1}$ in the interval $[\alpha,\beta]$ such that
$$
X^{(j+l)}_n(t_l)=0 \text{ for all } l\in\{0,\dots, m-1\}.
$$
Thus we may consider the random variable
\begin{equation*}
Y^{(j)}_n:=\IND{D^{(j)}_m}\times
\INT{t_0}{\beta}\INT{t_1}{x_1}\dots \INT{t_{m-1}}{x_{m-1}} X^{(j+m)}_n(x_m)\D x_m\dots \d x_1.
\end{equation*}
On the event $D^{(j)}_m$, the random variables $X^{(j)}_n(\beta)$ and $Y^{(j)}_n$ are equal. On the complement of $D^{(j)}_m$, $Y^{(j)}_n=0$.  Hence, it follows that
\begin{equation*}
\P{D_m^{(j)}\cap\left\{\frac{|X_n^{(j)}(\beta)|}{n^j}\geq T\right\}}\leq
\P{\frac{|Y_n^{(j)}|}{n^j}\geq T}.
\end{equation*}
Markov's inequality yields
\begin{align*}
\P{|Y^{(j)}_n|\geq Tn^{j}}
%&\leq \frac{1}{T^2n^{2j}}\VAR{Y_n^{(j)}} \\
\leq\frac{1}{T^2n^{2j}}\E
\left|\int_{t_0}^{\beta}\int_{t_1}^{x_1}\dots \int_{t_{m-1}}^{x_{m-1}}X_n^{(j+m)}(x_m)\D x_m\dots\d x_1\right|^2.
\end{align*}
Using Hölder's inequality we may proceed as follows
\begin{align*}
\P{|Y^{(j)}_n|\geq Tn^{j}}&\leq \frac{1}{T^2n^{2j}}\frac{(\beta-\alpha)^m}{m!}
\E
\int_{t_0}^{\beta}\int_{t_1}^{x_1} \dots \int_{t_{m-1}}^{x_{m-1}}|X^{(j+m)}_n(x_m)|^2\D x_m\dots \d x_1  \\
&\leq
\frac{1}{T^2n^{2j}}\left[\frac{(\beta-\alpha)^m}{m!}\right]^2
\sup_{x\in[\alpha,\beta]} \mathbb{E}|X_n^{(j+m)}(x)|^2.
\end{align*}
It remains to find a suitable estimate for $\sup_{x\in [\alpha,\beta]} \mathbb{E}|X^{(j+m)}_n(x)|^2$.
From Lemma~\ref{lemma:varianzderableitung} it follows that %for all $x\in[0,2\pi]$, $n\in \N$, $j\in\N_0$ and $m\in\N$
\begin{equation*}
\mathbb{E}|X^{(m+j)}_n(x)|^2
=
\VAR X^{(j+m)}_n(x)
=
\frac{1}{n}\sum_{k=1}^n k^{2(j+m)}\leq C(j,m) n^{2(j+m)}
\end{equation*}
holds, whence the statement follows immediately.
\end{proof}
%%%%%%%%%%%%%%%%%%%%%%%%%%%%%%%%%%%%%%%%%%%%%%%%%%%%%%%%%%%%%%%%%%%%%%%%%%%%%%%%%%%%%%%%%%%%%%%%%%%
\begin{lemma}\label{lemma:abschaetzungdrei}
Fix $j\in\N_0$. There exists a constant $C=C(j)>0$ such that for all $n\in\N$, $T>0$, $\beta\in[0,2\pi]$,
\begin{equation} \label{eq:concentration}
\P{\left|\frac{X^{(j)}_n(\beta)}{n^j}\right|\leq T}\leq
C\left(T+T^{-1/2}n^{-(2j+1)/4}\right).
\end{equation}
\end{lemma}
%%%%%%%%%%%%%%%%%%%%%%%%%%%%%%%%%%%%%%%%%%%%%%%%%%%%%%%%%%%%%%%%%%%%%%%%%%%%%%%%%%%%%%%%%%%%%%%%%%%
\begin{proof}
For $\lambda >0$ let $\eta$ be a random variable (independent of $X_n^{(j)}(\beta)$) with characteristic function
\begin{equation*}
\psi(t):=\EW{\EXP{\I t\eta}}=\frac{\sin^2(t\lambda)}{t^2\lambda^2}.
\end{equation*}
That is, $\eta$ is the sum of two independent random variables which are uniformly distributed on $[-\lambda,\lambda]$.
Consider the random variable
$$
\tilde{X}^{(j)}_n(\beta):=n^{-j} X^{(j)}_n(\beta)+\eta.
$$
For all $T>0$ we have
\begin{equation}\label{eq:abschaetzungdrei}
\P{\left|\frac{X^{(j)}_n(\beta)}{n^j}\right|\leq T}\leq
\P{|\tilde{X}^{(j)}_n(\beta)|\leq \frac{3}{2}T} +\P{|\eta|\geq \frac{1}{2}T}
\end{equation}
and we estimate the terms on the right-hand side separately.

\vspace*{2mm}
\noindent
\textit{First term on the RHS of~\eqref{eq:abschaetzungdrei}}. The density of $\tilde{X}_n^{(j)}(\beta)$ exists and can be expressed using the inverse Fourier transform of its characteristic function denoted in the following by
$$
\tilde{\phi}_n(t):=\E\EXP{\I t \tilde{X}_n^{(j)}(\beta)}.
$$
Using the representation for $X_n^{(j)}(\beta)$ obtained in the proof of Lemma~\ref{lemma:varianzderableitung} and recalling that $\varphi$ is the characteristic function of $A_k$ and $B_k$, we obtain
$$
|\tilde{\phi}_n(t)|
=\psi(t) \prod_{k=1}^n\left|\phi\left(k^j\frac{t\cos(k \beta)}{n^{j+1/2}}\right)\right|\left|\phi\left(k^j\frac{t\sin(k\beta)}{n^{j+1/2}}\right)\right|.
$$
Using Fourier inversion, for every $y\geq 0$ we may write
\begin{align*}
\P{|\tilde{X}^{(j)}_n(\beta)|\leq y}
%&=\INT{-y}{y} \frac{1}{2\pi}\INT{-\infty}{\infty} \tilde{\phi}_n(t) \e^{-\I t x}\D t \D x \\
&=\frac{2}{\pi} \INT{0}{\infty} \frac{\sin(yt)}{t}
\RE \tilde{\phi}_n(t) \D t\\
&\leq
\frac{2y}{\pi}  \INT{0}{\infty} \psi(t)
\prod_{k=1}^n \left|\phi\left(k^j\frac{t\cos(k \beta)}{n^{j+1/2}}\right)\right|\left|\phi\left(k^j\frac{t\sin(k\beta)}{n^{j+1/2}}\right)\right|
\D t.
\end{align*}
We used that $|t^{-1}\sin(yt)|\leq y$ for every $y\geq 0$ and $t\neq 0$.
The coefficients $A_k$ and $B_k$ are supposed to have zero mean and unit variance. From this we can conclude that
\begin{equation}\label{eq:varphi_est}
|\phi(t)|\leq \exp(-t^2/4) \text{ for } t\in [-c,c],
\end{equation}
where $c>0$ is a sufficiently small constant.
Let $\{\Gamma_l:l=0,\dots, n\}$ be a disjoint partition of $\R_+$ defined by
\begin{align*}
\Gamma_{l}&:=\left\{
t:\frac{cn^{j+1/2}}{(l+1)^j}\leq t<\frac{cn^{j+1/2}}{l^j}
\right\}
\quad \textrm{for }l=1,\dots, n-1,\\
\Gamma_{n}&:=\left\{t:0\leq t< c\sqrt{n}\right\},\\
\Gamma_{0}&:=\{t:t\geq cn^{j+1/2}\}.
\end{align*}
We decompose the integral above as follows:
$$
\P{|\tilde{X}^{(j)}_n(\beta)|\leq y} \leq \frac{2y}{\pi}\sum_{l=0}^n I_l,
$$
where
$$
I_l := \int_{\Gamma_l} \psi(t) \prod_{k=1}^n \left|\phi\left(k^j\frac{t\cos(k \beta)}{n^{j+1/2}}\right)\right|\left|\phi\left(k^j\frac{t\sin(k\beta)}{n^{j+1/2}}\right)\right|\D t.
$$
For the integral over $\Gamma_0$ we may write using
$|\phi(t)|\leq 1$ and $\sin^2(\lambda t)\leq 1$,
\begin{equation*}
I_0 \leq
\INT{cn^{j+1/2}}{\infty}
\psi(t) \D t=
\INT{cn^{j+1/2}}{\infty}
\frac{\sin^2(\lambda t)}{\lambda^2t^2} \D t \leq
\frac{1}{c\lambda^2}
n^{-(j+1/2)}.
\end{equation*}
The integral over $\Gamma_n$ is smaller than a positive constant $C>0$ independent of $n$ because we can estimate all terms involving $\varphi$ by means of~\eqref{eq:varphi_est} as follows:
\begin{equation*}
I_n\leq
\INT{0}{c\sqrt{n}}\psi(t)
\EXP{-\frac{1}{4}\frac{t^2}{n^{2j+1}}\sum_{k=1}^n k^{2j}}\D t
\leq \INT{0}{\infty} \EXP{-t^2\gamma}\D t
\leq C,
\end{equation*}
where $\gamma>0$ is a small constant and we used that
\begin{equation*}
\sum_{k=1}^n k^{2j}\sim \frac{n^{2j+1}}{2j+1} \quad \textrm{as }n\rightarrow\infty.
\end{equation*}
%and furthermore we used the estimate $|\phi(t)|\leq\EXP{-t^2/4}$ for $0\leq t\leq c$.
For $t\in\Gamma_l$ with $l=1,\dots, n-1$ we have
\begin{equation*}
\left|l^j\frac{t\cos(l\beta)}{n^{j+1/2}}\right| \leq \frac{tl^j}{n^{j+1/2}} \leq c,
\quad
\left|l^j\frac{t\sin(l\beta)}{n^{j+1/2}}\right| \leq \frac{tl^j}{n^{j+1/2}} \leq c.
\end{equation*}
Thus, we can estimate all factors with $k=1,\ldots,l$ using~\eqref{eq:varphi_est}, whereas for all other factors we use the trivial estimate $|\varphi(t)|\leq 1$:
%Every characteristic function is bounded by $1$ and thus for the remaining parts of the integral we obtain for $l=1,\dots, n-1$.
\begin{align*}
I_l&\leq
\int_{\Gamma_l} \psi(t)\EXP{-\frac{1}{4}\frac{t^2}{n^{2j+1}}\sum_{k=1}^l k^{2j}}
\D t \\
&\leq
\int_{\frac{cn^{j+1/2}}{(l+1)^j}}^{\frac{cn^{j+1/2}}{l^j}}
\frac{1}{\lambda^2t^2}
\EXP{-\gamma_1 t^2\left(\frac{l}{n}\right)^{2j+1}} \D t \\
&=\frac{1}{\lambda^2} \left(\frac{l}{n}\right)^{j+1/2}
\int_{c\frac{l^{j+1/2}}{(l+1)^j}}^{c\sqrt{l}}\frac{1}{u^2}
\EXP{-\gamma_1u^2}\D u \\
&\leq \frac{C}{\lambda^2} \left(\frac{l}{n}\right)^{j+1/2}\EXP{-\gamma_2 l},
\end{align*}
where $\gamma_1,\gamma_2>0$ are small constants and we substituted $u^2=t^2(l/n)^{2j+1}$. Summing up yields
\begin{equation*}
\sum_{l=1}^{n-1} I_l \leq 
C \lambda^{-2}n^{-(j+1/2)}\sum_{l=1}^{n-1}l^{j+1/2}\EXP{-\gamma_2l}
\leq C'\lambda^{-2}n^{-(j+1/2)}.
\end{equation*}
Taking the estimates for $I_0,\ldots,I_n$ together, for every $y\geq 0$ we obtain
\begin{equation}\label{eq:beweisesteins}
\P{|\tilde{X}_n^{(j)}(\beta)|\leq y}\leq
Cy\left(\frac{1}{\lambda^2}n^{-(j+1/2)}+1\right).
\end{equation}

\vspace*{2mm}
\noindent
\textit{Second term on the RHS of~\eqref{eq:abschaetzungdrei}}.
The second term on the right hand-side of \eqref{eq:abschaetzungdrei} can be estimated using Chebyshev's inequality (and $\E \eta=0$). Namely, for every $z>0$,
\begin{equation}\label{eq:beweisestzwei}
\P{|\eta|\geq z} \leq \frac{\VAR{\eta}}{z^2}
= \frac{2}{3}\frac{\lambda^2}{z^2}.
\end{equation}

\vspace*{2mm}
\noindent
\textit{Proof of~\eqref{eq:concentration}}.
We arrive at the final estimate setting $y=3T/2$ and $z=T/2$ in \eqref{eq:beweisesteins} and \eqref{eq:beweisestzwei} respectively. We obtain that for every $\lambda>0$ and $T>0$ the inequality
\begin{equation*}
\P{\left|\frac{X^{(j)}_n(\beta)}{n^j}\right|\leq T}\leq
C\left(\frac{T}{\lambda^2}n^{-(j+1/2)}+T+\frac{\lambda^2}{T^2}\right)
\end{equation*}
holds for a positive constant $C=C(j)>0$. This bound can be optimized by choosing a suitable $\lambda>0$.
Setting $\lambda=T^{3/4} n^{-(j/4+1/8)}$ the statement of the lemma follows.
\end{proof}
%%%%%%%%%%%%%%%%%%%%%%%%%%%%%%%%%%%%%%%%%%%%%%%%%%%%%%%%%%%%%%%%%%%%%%%%%%%%%%%%%%%%%%%%%%%%%%%%%%%

\subsection{Roots and sign changes}
The next lemma contains the main result of this section.
\begin{lemma}\label{lemma:ewuntnsternn}
For every $\delta\in (0,1/2)$ there exists $n_0=n_0(\delta)\in \N$ such that for all $n\geq n_0$ and
every interval $[\alpha,\beta]\subset [0,2\pi]$ of length $\beta-\alpha=\delta n^{-1}$ we have the estimate
\begin{equation*}
0\leq \E N_n[\alpha,\beta]- \E N_{n}^*[\alpha,\beta] \leq
 C (\delta^{4/3}+\delta^{-7}n^{-1/4}),
\end{equation*}
where $C>0$ is a constant independent of $n$, $\delta$, $\alpha$, $\beta$.
\end{lemma}
A crucial feature of this estimate is that the exponent $4/3$ of $\delta$ is $>1$, while the exponent of $n$ is negative.
%--------------------------------------------------------------------------------------------------
\begin{proof}
Let $D^{(j)}_m$ be the random event defined as in Section~\ref{subsec:D_m_j}.
Observe that due to the convention in which way $N_n[\alpha,\beta]$ counts the roots, the difference between $N_n^*[\alpha,\beta]$ and $N_n[\alpha,\beta]$ vanishes in the following cases:
\begin{itemize}
\item $X_n$ has no roots in $[\alpha,\beta]$ (in which case $N_n[\alpha,\beta]=N_n^*[\alpha,\beta]=0$);
\item $X_n$ has exactly one simple root in $(\alpha,\beta)$ and no roots on the boundary (in which case $N_n[\alpha,\beta]=N_n^*[\alpha,\beta]=1$);
\item $X_n$ has no roots in $(\alpha,\beta)$ and one simple root (counted as $1/2$) at either $\alpha$ or $\beta$ (in which case $N_n[\alpha,\beta]=N_n^*[\alpha,\beta]=1/2$).
\end{itemize}
In all other cases (namely, on the event $D^{(0)}_2$ when the number of roots in $[\alpha,\beta]$, with multiplicities, but without $1/2$-weights on the boundary, is at least $2$) we only have the trivial estimate
$$
0\leq N_n[\alpha,\beta]-N_n^*[\alpha,\beta]\leq N_n[\alpha,\beta].
$$
Since $D_2^{(0)}\supseteq D_3^{(0)}\supseteq \ldots$ and on the event $D^{(0)}_m \backslash D^{(0)}_{m+1}$ it holds that $N_n[\alpha,\beta]\leq m$, we obtain
\begin{align*}
0\leq \E N_n[\alpha,\beta]-\E N_n^*[\alpha,\beta]
&\leq
\EW{N_n[\alpha,\beta]\IND{D^{(0)}_2}}\\
&\leq\P{D^{(0)}_2}+\sum_{m=2}^{2n}\P{D^{(0)}_m} \\
&\leq
\P{D^{(0)}_2}+\sum_{m=2}^{21}\P{D^{(0)}_m}
+\sum_{m=2}^{2n-20} \P{D^{(20)}_{m}},
\end{align*}
where in the last step we passed to the $20$-th derivative of $X_n$ using Rolle's theorem.
The upper bounds for  the first two terms on the right-hand side follow immediately by Lemma
\ref{lemma:abschaetzungdmj}, namely
$$
\P{D^{(0)}_2} + \sum_{m=2}^{21}\P{D^{(0)}_m} \leq C(\delta^{4/3} + \delta^{-7} n^{-1/4}).
$$
Thus we focus on the last term. For every $\delta>0$ (and $n$ big enough) we can find a number $k_0 = k_0(\delta,n)\in\{2,\dots, 2n\}$ such that

\begin{equation*}
n^{2}\leq\delta^{-k_0/3}<\delta^{-2k_0/3} \leq n^{5}.
\end{equation*}
For $m=2,\dots, k_0$ the estimate for the probability of $D^{(20)}_m$ presented in Lemma \ref{lemma:abschaetzungdmj} is good enough, whereas  for $m=k_0+1,\dots, 2n-20$ we use the fact that  $D^{(20)}_{k_0}\supseteq D^{(20)}_{k_0+l}$ for all $l\in \N$. This yields
\begin{align*}
\sum_{m=2}^{2n-20} \P{D^{(20)}_{m}}
&\leq
\sum_{m=2}^{k_0}\P{D^{(20)}_{m}}
+\sum_{m=k_0+1}^{2n-20}\P{D^{(20)}_{k_0}} \\
& \leq
\sum_{m=2}^{k_0} C(\delta^{2m/3}+ \delta^{-m/3} n^{-10})
+ 2Cn (\delta^{2k_0/3} + \delta^{-k_0/3}n^{-10})\\
&\leq
C(\delta^{4/3}+n^{-5})+2Cn (n^{-2} + n^{-5})\\
&\leq
C(\delta^{4/3}+\delta^{-7}n^{-1/4}).
\end{align*}
Combining the above estimates yields the statement of the lemma.
\end{proof}

%-----------------------------------------------------------------------------------------------------------------------
\section{The related stationary Gaussian process}\label{abschnitt:berechnung}

\subsection{Convergence to the Gaussian case}
In the following let $(Z(t))_{t\in\R}$ denote the stationary Gaussian process with $\E Z(t)=u$, $\VAR Z(t)=1$, and covariance
\begin{equation*}
\Cov\left[Z(t),Z(s)\right]=\frac{\sin(t-s)}{t-s}, \quad t\neq s.
\end{equation*}
The following lemma states the weak convergence of the bivariate distribution of $(X_n(\alpha),X_n(\beta))$ with $\beta-\alpha = n^{-1}\delta$ to $(Z(0), Z(\delta))$, as $n\to\infty$.
%%%%%%%%%%%%%%%%%%%%%%%%%%%%%%%%%%%%%%%%%%%%%%%%%%%%%%%%%%%%%%%%%%%%%%%%%%%%%%%%%%%%%%%%%%
\begin{lemma}\label{lem:lim_S_n}
Let $\delta>0$ be arbitrary but fixed. For $n\in\N$ let $[\alpha_n,\beta_n]\subseteq [0,2\pi]$ be an interval of length $\beta_n-\alpha_n=n^{-1}\delta$. Then
\begin{equation*}
\begin{pmatrix}
X_n(\alpha_n)\\
X_n(\beta_n)
\end{pmatrix}
\to
\begin{pmatrix}
Z(0) \\
Z(\delta)
\end{pmatrix}
\quad \textrm{in distribution as $n\to\infty$}.
\end{equation*}
\end{lemma}
%----------------------------------------------------------------------------------------
\begin{proof}
To prove the statement it suffices to show the pointwise convergence of the corresponding characteristic functions. Let
\begin{equation*}
\phi_n(\lambda,\mu):=\E \e^{\I(\lambda X_n(\alpha_n)+\mu X_n(\beta_n))}
\end{equation*}
denote the characteristic function of $(X_n(\alpha_n),X_n(\beta_n))$. Recall that $\phi$ represents the common characteristic function of the coefficients $(A_k)_{k\in\N}$ and $(B_k)_{k\in\N}$. Then the expression reads
\begin{align*}
&\phi_n(\lambda,\mu)= \\
&\e^{\I u(\lambda+\mu)}
\prod_{k=1}^n
\phi\left(\frac{\lambda \cos(k\alpha_n)+\mu\cos(k\beta_n)}{\sqrt{n}}\right)
\phi\left(\frac{\lambda \sin(k\alpha_n)+\mu\sin(k\beta_n)}{\sqrt{n}}\right).
\end{align*}
Using \eqref{eq:charfuncchar} we have
\begin{align*}
\phi_n(\lambda,\mu)&=\e^{-S_n(\lambda,\mu)}, \\
S_n(\lambda, \mu)&:=
-\I u(\lambda+\mu)+\frac{1}{2n}\sum_{k=1}^n (\lambda\cos(k\alpha_n)+\mu\cos(k\beta_n ))^2H_1(n,k)\\
&+\frac{1}{2n}\sum_{k=1}^n (\lambda\sin(k\alpha_n)+\mu\sin(k\beta_n))^2H_2(n,k),
\end{align*}
where we have shortened the writing by defining
\begin{align*}
H_1(n,k)&:=H\left(\frac{\lambda\cos(k\alpha_n)+\mu\cos(k\beta_n)}{\sqrt{n}}\right),\\
H_2(n,k)&:=H\left(\frac{\lambda\sin(k\alpha_n)+\mu\sin(k\beta_n)}{\sqrt{n}}\right).
\end{align*}
After elementary transformations and using that $\beta_n-\alpha_n=n^{-1}\delta$ we obtain
\begin{align*}
& S_n(\lambda, \mu)= \\
&-\I u(\lambda+\mu) +\frac{1}{n}\sum_{k=1}^n H_1(n,k)
\left(
\frac{\lambda^2}{2}+\frac{\mu^2}{2}+\lambda\mu
\cos\left(k\frac{\delta}{n}\right)
\right)
+R_n(\lambda,\mu),
\end{align*}
where we have abbreviated
\begin{equation*}
R_n(\lambda,\mu):=\frac{1}{2n}\sum_{k=1}^n(\lambda\sin(k\alpha_n)+\mu\sin(k\beta_n ))^2(H_2(n,k)-H_1(n,k)).
\end{equation*}
Since Riemann sums converge to Riemann integrals, we have
$$
\lim_{n\to\infty} \frac{1}{n}\sum_{k=1}^n 
\left(
\frac{\lambda^2}{2}+\frac{\mu^2}{2}+\lambda\mu
\cos\left(k\frac{\delta}{n}\right)
\right)
=
\frac{\lambda^2}{2}+\frac{\mu^2}{2}+\lambda\mu \frac{\sin \delta}{\delta}. 
$$
For $i=1,2$ we have that $\lim_{n\to\infty} H_i(n,k) = H(0)=1$ uniformly in $k=1,2,\dots, n$. Hence,
\begin{equation*}
\left|\frac{1}{n}\sum_{k=1}^n (H_1(n,k)-1)\left(
\frac{\lambda^2}{2}+\frac{\mu^2}{2}+\lambda\mu
\cos\left(k\frac{\delta}{n}\right)
\right)\right|
\leq 
\frac{C}{n}\sum_{k=1}^n |H_1(n,k)-1|
\longrightarrow 0
\end{equation*}
as $n\to\infty$. The remaining term of the sum
\begin{equation*}
 |R_n(\lambda,\mu)|\leq \frac{1}{2n}\sum_{k=1}^nC|H_2(n,k)-H_1(n,k)|
\longrightarrow 0 %\quad \textrm{for } n\rightarrow\infty
\end{equation*}
goes to $0$ for all fixed $\lambda,\mu$, as $n\to \infty$. Therefore we have
\begin{equation}\label{eq:S_infty}
S_\infty(\lambda,\mu):=\lim_{n\rightarrow\infty} S_n(\lambda, \mu)=
-\I u(\lambda+\mu)+
 \frac{\lambda^2+\mu^2}{2} + \lambda\mu \frac{\sin(\delta)}{\delta}
\end{equation}
and $\phi_\infty(\lambda,\mu):= \EXP{-S_\infty(\lambda,\mu)}$ is nothing but the characteristic function of $(Z(0),Z(\delta))$. This implies the statement.
\end{proof}
%%%%%%%%%%%%%%%%%%%%%%%%%%%%%%%%%%%%%%%%%%%%%%%%%%%%%%%%%%%%%%%%%%%%%%%%%%%%%%%%%%%%%%%%%%

\subsection{The Gaussian case}
Denote by $\tilde N^*[\alpha,\beta]$ the analogue of $N_n^*[\alpha, \beta]$ for the process $Z$, that is
\begin{equation}\label{eq:definitionnstern}
\tilde N^*[\alpha,\beta]:=\frac{1}{2}-\frac{1}{2}\sgn(Z(\alpha)Z(\beta)).
\end{equation}
\begin{lemma}\label{lem:gaussian}
%Uniformly over all intervals $[\alpha,\beta] \subseteq [0,2\pi]$ of length $\beta-\alpha = n^{-1} \delta$ it holds that
As $\delta\downarrow 0$, we have
\begin{equation}\label{eq:gaussian_crossing_probab}
\E \tilde{N}^*[0,\delta]= \frac{1}{\pi\sqrt{3}}\EXP{-\frac{u^2}{2}}\delta+o(\delta).
\end{equation}
\end{lemma}
\begin{proof}
The bivariate random vector $(Z(0),Z(\delta))$ is normal-distributed with mean $(u,u)$ and covariance $\rho=\delta^{-1}\sin \delta$. We have
\begin{align*}
\E \tilde{N}^*[0,\delta]
&=\P{Z(0)Z(\delta)<0}\\
&=2\P{Z(0)-u < -u, Z(\delta)-u>-u}\\
&\sim\frac{\sqrt{1-\rho^2}}{\pi}\EXP{-\frac{u^2}{2}}
\end{align*}
as $\delta\downarrow 0$ (equivalently, $\rho\uparrow 1$), where the last step will be justified in  Lemma~\ref{lemma:levelcrossings}, below.
Using the Taylor series of $\rho^{-1}\sin \rho$ which is given by
\begin{equation}
\frac{\sin(\delta)}{\delta} = 1-\frac{\delta^2}{6}+o(\delta^2) \quad \textrm{as }
\delta \downarrow 0,
\end{equation}
we obtain the required relation~\eqref{eq:gaussian_crossing_probab}.
\end{proof}

\begin{lemma}\label{lemma:levelcrossings} Let $(X,Y)\sim N(\mu,\Sigma)$ be bivariate normal distributed with parameters
\begin{equation*}
\mu=\begin{pmatrix}
0 \\
0
\end{pmatrix}
 \quad \textrm{und} \quad
\Sigma=\begin{pmatrix}
1 & \rho \\
\rho & 1
\end{pmatrix}.
\end{equation*}
Let $u\in \R$ be arbitrary but fixed. Then,
\begin{equation*}
\P{X\leq u,Y\geq u}\sim \frac{\sqrt{1-\rho^2}}{2\pi}\EXP{-\frac{u^2}{2}}
\qquad \textrm{as } \rho \uparrow 1.
\end{equation*}
\end{lemma}

%--------------------------------------------------------------------------------------------------
\begin{proof}
In the special case $u=0$ the lemma could be deduced from the explicit formula
$$
\P{X\geq 0,Y\geq 0} = \frac 14 + \frac{\arcsin \rho}{2\pi}
$$
due to F.\ Sheppard; see~\cite{bingham_doney} and the references therein. For general $u$, no similar formula seems to exist and we need a different method.

By the formula for the density of the random vector $(X,Y)$, we have to investigate the integral
\begin{equation*}
\int_{x\leq u} \int_{y\geq u} \frac{1}{2\pi\sqrt{(1-\rho^2)}}
\exp\left(-\frac{1}{2(1-\rho^2)}(x^2+y^2-2\rho xy)\right)\ \d x \d y
\end{equation*}
as $\rho\rightarrow 1$.
After the substitution $x=u-\epsilon v$ and $y=u+\epsilon w$ with a parameter $\epsilon>0$ to be chosen below, the integral becomes
\begin{align*}
&\frac{\epsilon^2}{2\pi\sqrt{1-\rho^2}}\exp\left(-\frac{u^2}{1+\rho}\right) \\
&\quad \times \INT{0}{\infty}\INT{0}{\infty} \EXP{
\frac{u\epsilon}{1+\rho}(v-w) -\frac{\epsilon^2}{2(1-\rho^2)}(v^2+w^2+2\rho vw)
}\mathrm{d}v\mathrm{d}w.
\end{align*}
Setting $\epsilon:=\sqrt{1-\rho^2}$ we have $\epsilon\rightarrow 0$ for $\rho\rightarrow 1$ and furthermore (using the dominated convergence theorem)
\begin{align*}
\P{X\leq u,Y\geq u}&\sim \frac{\sqrt{1-\rho^2}}{2\pi}\exp\left(-\frac{u^2}{2}\right)
\INT{0}{\infty}\INT{0}{\infty}
\exp\left(-\frac{1}{2}(v+w)^2\right)\mathrm{d}v\mathrm{d}w \\
&=\frac{\sqrt{1-\rho^2}}{2\pi}\exp\left(-\frac{u^2}{2}\right) \quad \textrm{as }\rho \rightarrow 1,
\end{align*}
where we have used that
\begin{equation*}
\INT{0}{\infty}\INT{0}{\infty} \exp\left(-\frac{1}{2}(v+w)^2\right)\D v\d w =
\INT{0}{\infty}z\EXP{-\frac{1}{2}z^2}\D z =1.
\end{equation*}
\end{proof}

%-----------------------------------------------------------------------------------------------------------------------
\section{Proof of the main result}
\subsection{Approximation by a lattice}
Fix an interval $[a,b]\subset [0,2\pi]$ and take some $\delta>0$. We will study the sign changes of $X_n$ on the lattice $\delta n^{-1} \Z$. Unfortunately, the endpoints of the interval $[a,b]$ need not be elements of this lattice.
To avoid boundary effects, we notice that $[a_n', b_n']\subset [a,b] \subset [a_n,b_n]$ with
\begin{equation*}
a_n:=\frac{\delta}{n}\left\lfloor \frac{an}{\delta} \right\rfloor,
\quad
b_n:=\frac{\delta}{n}\left\lceil \frac{bn}{\delta} \right\rceil, 
\quad
a'_n:=\frac{\delta}{n}\left\lceil \frac{an}{\delta} \right\rceil,
\quad
b'_n:=\frac{\delta}{n}\left\lfloor \frac{bn}{\delta} \right\rfloor.
\end{equation*}
Since $N_n[a_n',b_n']\leq N_n[a,b]\leq N_n[a_n,b_n]$, it suffices to show that 
$$
\lim_{n\to\infty} \frac{\E N_n[a_n,b_n]}{n} = \lim_{n\to\infty} \frac{\E N_n[a_n',b_n']}{n} =\frac{b-a}{\pi\sqrt{3}}\EXP{-\frac{u^2}{2}}.
$$
In the following, we compute the first limit because the second one is completely analogous.

Let $N^*_{n,\delta}[a_n,b_n]$ be a random variable counting the number of sign changes of $X_n$ on the lattice $\delta n^{-1}\Z$ between $a_n$ and $b_n$, namely
\begin{equation*}
N^*_{n,\delta}[a_n,b_n]:=\sum_{k=\lfloor \delta^{-1}an\rfloor}^{\lceil\delta^{-1}bn\rceil-1}
N_n^*\left[\frac{k\delta}{n},\frac{(k+1)\delta}{n}\right].
\end{equation*}
The following lemma claims that the expected difference between the number of roots and the number of sign changes is asymptotically small.
%%%%%%%%%%%%%%%%%%%%%%%%%%%%%%%%%%%%%%%%%%%%%%%%%%%%%%%%%%%%%%%%%%%%%%%%%%%%%%%%%%%%%%%%%%
\begin{lemma}\label{lemma:hauptbeweiszwei}
It holds that
\begin{equation*}
\lim_{\delta\downarrow 0}\limsup_{n\to\infty} \frac{\E N_{n}[a_n,b_n] - \E N^*_{n,\delta}[a_n,b_n]}{n}=0.
\end{equation*}
\end{lemma}
\begin{proof}
The triangle inequality and Lemma \ref{lemma:ewuntnsternn} yield that for any $\delta\in (0,1/2)$ and all sufficiently large $n$, 
\begin{align*}
&\left|\E N_n[a_n,b_n]-\E N^*_{n,\delta}[a_n,b_n]\right| \\
=&\quad
\left|\sum_{k=\lfloor \delta^{-1}an\rfloor}^{\lceil\delta^{-1}bn\rceil-1}
\left(\E N_n\left[\frac{k\delta}{n},\frac{(k+1)\delta}{n}\right] -
\E N^*_n\left[\frac{k\delta}{n},\frac{(k+1)\delta}{n}\right]\right)\right| \\
\leq&\quad
\sum_{k=\lfloor \delta^{-1}an\rfloor}^{\lceil\delta^{-1}bn\rceil-1}
\left|\E N_n\left[\frac{k\delta}{n},\frac{(k+1)\delta}{n}\right] -
\E N^*_n\left[\frac{k\delta}{n},\frac{(k+1)\delta}{n}\right]\right| \\
\leq &\quad
C\frac{n}{\delta}\left(\delta^{4/3}+\delta^{-7}n^{-1/4}\right).
\end{align*}
It follows that for every fixed $\delta\in (0,1/2)$,
\begin{equation*}
\limsup_{n\to\infty} \frac{\left|\E N_n[a_n,b_n]-\E N^*_{n,\delta}[a_n,b_n]\right|}{n}
\leq C\delta^{1/3}.
\end{equation*}
Letting  $\delta \to 0$ completes the proof.
\end{proof}

\subsection{Sign changes over the lattice}
In the next lemma we find the asymptotic  number of sign changes of $X_n$ over a lattice with mesh size $\delta n^{-1}$, as $n\to\infty$ and then $\delta\downarrow 0$. 
%%%%%%%%%%%%%%%%%%%%%%%%%%%%%%%%%%%%%%%%%%%%%%%%%%%%%%%%%%%%%%%%%%%%%%%%%%%%%%%%%%%%%%%%%%
\begin{lemma}\label{lemma:hauptbeweisdrei}
%For all $[a,b]\subseteq [0,2\pi]$ it applies that
It holds that
\begin{equation*}
\lim_{\delta\downarrow 0} \lim_{n\to\infty}
\frac{\E N^*_{n,\delta}[a_n,b_n]}{n}=\frac{b-a}{\pi\sqrt{3}}\EXP{-\frac{u^2}{2}}.
\end{equation*}
\end{lemma}
%--------------------------t---------------------------------------------------------------
\begin{proof}
Fix $0< \delta\leq 1$.
For every $x\in\R$ we can find a unique integer $k=k(x;n,\delta)$ such that $x\in (n^{-1}\delta k,n^{-1}\delta (k+1)]$. Thus the function $f_n: \R\rightarrow [0,1]$ given by
\begin{equation*}
f_n(x):=\E N_n^*\left[\frac{k\delta}{n},\frac{(k+1)\delta}{n}\right] \quad
\textrm{for} \quad \frac{k\delta}{n}<x\leq \frac{(k+1)\delta}{n}
\end{equation*}
is well-defined for all $n\in\N$. Now $\E N^*_{n,\delta}[a_n,b_n]$ can be expressed as
\begin{equation*}
\frac{\E N_{n,\delta}^*[a_n,b_n]}{n}=\frac{1}{\delta}\INT{a_n}{b_n}f_n(x)\D x.
\end{equation*}
Recall that $(Z(t))_{t\in\R}$ denotes the stationary Gaussian process with mean $\E Z(t)=u$ and covariance
\begin{equation*}
\Cov\left[Z(t),Z(s)\right]=\frac{\sin(t-s)}{t-s}.
\end{equation*}
We want to show that for all $x\in \R$,
\begin{equation}\label{eq:konvergenz}
\lim_{n\to\infty} f_n(x)=\P{Z(0)Z(\delta)\leq 0}.
\end{equation}
Write $\alpha_n:=k n^{-1}\delta$ and $\beta_n:=(k+1)n^{-1}\delta$, so that $\beta_n-\alpha_n=n^{-1}\delta$.  We obtain from Lemma~\ref{lem:lim_S_n} that
\begin{equation*}
\begin{pmatrix}
X_n(\alpha_n) \\
X_n(\beta_n)
\end{pmatrix}
\to
\begin{pmatrix}
Z(0) \\Z(\delta)
\end{pmatrix}
 \quad \textrm{in distribution as $n\to\infty$.}
\end{equation*}
Now consider the function
\begin{equation*}
h:\R^2\to \{-1,0,1\},\quad h(x,y):=\sgn(x)\sgn(y).
\end{equation*}
Let $D_h\subseteq \R^2$ be the set of discontinuities of $h$, which in this case is the union of the coordinate axes. Since $(Z(0),Z(\delta))^T$ is bivariate normal with unit variances, it follows that
\begin{equation*}
\P{
\begin{pmatrix}
Z(0) \\Z(\delta)
\end{pmatrix}
\in D_h}=0.
\end{equation*}
Using the continuous mapping theorem  (see, e.g., \cite[Theorem 2.7]{billingsley_book}), we conclude that
\begin{equation*}
h(X_n(\alpha_n),X_n(\beta_n)) \to h(Z(0), Z(\delta)) \quad \textrm{in distribution as $n\to\infty$}.
\end{equation*}
Since these random variables are bounded by $1$, it follows that
\begin{equation*}
\lim_{n\to\infty} \EW{\sgn X_n(\alpha_n)\sgn X_n(\beta_n)}=
\EW{\sgn Z(0)\sgn Z(\delta)}.
\end{equation*}
Recalling that
\begin{align*}
f_n(x)=\E N^*_n[\alpha_n,\beta_n]&=\EW{\frac{1}{2}-\frac{1}{2}\sgn X_n(\alpha_n)\sgn X_n(\beta_n)},\\
\P{Z(0)Z(\delta)\leq 0}&=\EW{\frac{1}{2}-\frac{1}{2}\sgn Z(0)\sgn Z(\delta)}
\end{align*}
completes the proof of~\eqref{eq:konvergenz}.

Since $0\leq f_n\leq 1$ for all $n\in \N$, we may use the dominated convergence theorem to obtain
\begin{align*}
\lim_{n\to\infty} \INT{a_n}{b_n} f_n(x)\D x =\INT{a}{b} \P{Z(0)Z(\delta)\leq 0} \D x
=(b-a)\P{Z(0)Z(\delta)\leq 0}.
\end{align*}
Therefore for all $\delta>0$,
\begin{equation*}
\lim_{n\to\infty} \frac{\E N_{n,\delta}^* [a_n,b_n]}{n}=
\frac{\P{Z(0)Z(\delta)\leq 0}}{\delta} (b-a).
\end{equation*}
Due to Lemma \ref{lem:gaussian},
\begin{equation*}
\lim_{\delta\downarrow 0}\frac{\P{Z(0)Z(\delta)\leq 0}}{\delta}=
\EXP{-\frac{u^2}{2}}\frac{1}{\pi\sqrt{3}},
\end{equation*}
whence the statement follows.
\end{proof}
%%%%%%%%%%%%%%%%%%%%%%%%%%%%%%%%%%%%%%%%%%%%%%%%%%%%%%%%%%%%%%%%%%%%%%%%%%%%%%%%%%%%%%%%%%
\noindent \emph{Proof of Theorem 1.}
The triangle inequality yields
\begin{align*}
&\quad \left|\frac{\E N_n[a_n,b_n]}{n}-\frac{b-a}{\pi\sqrt{3}}\EXP{-\frac{u^2}{2}}\right| \\
&\leq
\left|\frac{\E N_{n}[a_n,b_n] - \E N^*_{n,\delta}[a_n,b_n]}{n}\right|+
\left|\frac{\E N^*_{n,\delta}[a_n,b_n]}{n}
-\frac{b-a}{\pi\sqrt{3}}\EXP{-\frac{u^2}{2}}\right|.
\end{align*}
Taking first $n$ to infinity and $\delta>0$ to zero afterwards, the first term of the sum on the right-hand side converges to $0$ due to Lemma \ref{lemma:hauptbeweiszwei}, while the second term of the sum converges to $0$ due to Lemma \ref{lemma:hauptbeweisdrei}. This proves that
$$
\lim_{n\to\infty} \frac{\E N_n[a_n,b_n]}{n} = \frac{b-a}{\pi\sqrt{3}}\EXP{-\frac{u^2}{2}}.
$$
Analogous argument shows that $[a_n,b_n]$ can be replaced by $[a_n', b_n']$. This completes the proof. 
\beweisende
%-----------------------------------------------------------------------------------------------------------------------
\bibliographystyle{alpha}
\bibliography{random_trig}
\end{document}